\documentclass[reqno,a4paper,12pt]{amsart}

\usepackage{amssymb}
\numberwithin{equation}{section}
\date{\today}

\newcommand{\q}{\qquad}

\newcommand\x{\mathbf{x}}

\newcommand{\R}{\mathbb{R}}

\newcommand{\grad}{\nabla}
\newcommand{\cf}{\mathcal {F}}

\newcommand{\beq}{\begin{equation}}
\newcommand{\eeq}{\end{equation}}
\newcommand{\bdm}{\begin{displaymath}}
\newcommand{\edm}{\end{displaymath}}
\newcommand{\ba}{\begin{align}}
\newcommand{\ea}{\end{align}}

\newcommand{\id}{\mathbf{1}}                

\newtheorem{Lemma}{Lemma}
\newtheorem{Theorem}{Theorem}
\newtheorem{Corollary}{Corollary}

\newtheorem{Remark}{Remark}

\begin{document}

\title[Improved Hardy-Sobolev inequalities]{Improved Hardy-Sobolev
inequalities}
\author[A.~Balinsky]{A. Balinsky}
\address{School of Mathematics\\
Cardiff University\\
         23 Senghennydd Road\\
         Cardiff CF2 4YH\\
         UK}
\email{BalinskyA@cardiff.ac.uk}
\author[W.~D.~Evans]{W.~D. Evans}
\address{School of Mathematics\\
Cardiff University\\
         23 Senghennydd Road\\
         Cardiff CF2 4YH\\
         UK}
\email{EvansWD@cardiff.ac.uk}

\author[D.~Hundertmark]{D.~Hundertmark}
\address{School of Mathematics\\ University of Birmingham\\
Edgbaston, Birmingham B15 2TT\\ UK\\ On leave from Department of Mathematics\\
Altgeld Hall\\
         University of Illinois at Urbana-Champaign\\
         1409 W. Green Street\\
         Urbana, Illinois 61801\\
         USA}
\email{dirk@math.uiuc.edu}
\author[R.~T.~Lewis]{R.~T. Lewis}
\address{Department of Mathematics\\
         University of Alabama at Birmingham\\
         Birmingham, AL 35294-1170\\
         USA}
\email{lewis@math.uab.edu}


\thanks{The second author (WDE) gratefully acknowledges the hospitality and support of the
Isaac Newton Institute, University of Cambridge, during June,
2007, when some of this work was done. }
\thanks{The third author (DH) thanks the US National Science Foundation for financial support from grant DMS-0400940.}

\begin{abstract} The main result includes features of a Hardy-type
inequality and an inequality of either Sobolev or
Gagliardo-Nirenberg type. It is inspired by the method of proof of
a recent improved Sobolev inequality derived by M. Ledoux which
brings out the connection between Sobolev embeddings and heat
kernel bounds. Here Ledoux's technique is applied to the operator
$L:= {\bf{x}} \cdot \nabla $ and the analysis requires the
determination of the operator semigroup $\{e^{-tL^{*}L}\}_{t>0}$
and its properties.
\end{abstract}

\maketitle

\section{Introduction}
The best possible constant in Hardy's inequality
\begin{equation}\label{1.1}
    \int_{\R^n}|\nabla f|^p d\x \ge C(n,p)
    \int_{\R^n} \frac{|f(\x)|^p}{|\x|^p}d\x
\end{equation}
is $ C(n,p) = \{(n-p)/p\}^p$ and so the inequality only yields
non-trivial information when $n\neq p.$ In Theorem \ref{BE} below,
we prove that the related inequality
\begin{equation}\label{1.2}
    \int_{\R^n}|(\x\cdot\nabla) f(\x)|^p d\x \ge (n/p)^p
    \int_{\R^n}|f(\x)|^pd\x
\end{equation}
is satisfied for all values of $n$, including $n=p$, and this
implies Hardy's inequality for $1\le p \le n.$ The case $n=p$ has
a special significance also for the Sobolev inequality
\begin{equation}\label{1.3}
    \|f\|_{L^q(\R^n)} \le C'(n,p)\|\nabla f\|_{L^p(\R^n)},\ \ \  q = p^{*}=np/(n-p),\
    1\le p<n;
\end{equation}
when $n=p,$ (\ref{1.3}) does not hold for $q=\infty$. In
\cite{CDDV}, \cite{CDPX} and \cite{L}, the following improvement
of the Sobolev inequality is derived: for $1\le p<q<\infty,$
\begin{equation}\label{1.4}
\|f\|_{L^q(\R^n)} \le C'(n,p)\|\nabla f\|_{L^p(\R^n)}^{p/q}
\|f\|^{1-p/q}_{B_{\infty,\infty}^{p/(p-q)}}:
\end{equation}
the space $B_{\infty,\infty}^{p/(p-q)}$ is a Besov space defined
in terms of the heat semigroup $e^{t\Delta}$ (c.f.\cite{T},
Section 2.5.2). This includes, in particular, the Sobolev and
Gagliardo-Nirenberg inequalities, and also has important features
not possessed by (\ref{1.3}); see \cite{CDDV}, \cite{CDPX} and
\cite{L} for details.

This paper has two objectives: first to determine the semigroup
$e^{-tL^{*}L},$ where $L=\x\cdot\nabla$ in $L^2(\R^n),$ and then
to use this to derive an improved version of (\ref{1.2}) which is
analogous to (\ref{1.4}). A corollary of our main theorem in the
case $p=2$ is the inequality:
\begin{eqnarray}\label{1.5}
    \| r F(r)\|^2_{L^{2^{*}}(\R^+; d\mu))}&\le&
    C\left\{\|Lf\|^2_{L^2(\R^n)}-\frac{n^2}{4}\|f\|^2_{L^2(\R^n)}\right\}^{1/n}\nonumber
    \\
& \times &    \|f\|_{L^2(\R^n)}^{2(1-1/n)},
\end{eqnarray}
where $ 2^{*} = 2n/(n-2), d\mu(r) = r^{n-1}dr, C$ is a positive
constant depending only on $n$ and, in polar co-ordinates
$\x=r\omega, F(r)$ is the integral mean of $f$ over the unit
sphere $\mathbb{S}^{n-1},$ that is,
\[
F(r) := \frac{1}{|\mathbb{S}^{n-1}|} \int_{\mathbb{S}^{n-1}}
f(r\omega) d\omega.
\]
This has a number of consequences. One is a Hardy-Sobolev type
inequality (Corollary 5) which is that if $ f, \grad f \in
L^2(\R^n), n\ge 3,$ then,
\begin{eqnarray*}
\|F(r)\|^2_{L^{2^{*}}(\R^+; d\mu)} &\le &
    C\left\{\|\nabla
    f\|^2_{L^2(\R^n)}-\frac{(n-2)^2}{4}\|f/|\cdot|\|^2_{L^2(\R^n)}\right\}^{1/n}\nonumber
    \\
    & \times & \|f/|\cdot|\|_{L^2(\R^n)}^{2(1-1/n)}
\end{eqnarray*}
which yields, for any $\delta \in [0,(n-2)^2/4),$
\begin{equation}\label{Stubbepre}
    \|F\|^2_{L^{2^{*}}(\R^+; d\mu)} \le C [\frac{(n-2)^2}{4} -\delta]^{-\frac{(n-1)}{n}}\left\{\|\nabla
    f\|^2_{L^2(\R^n)}- \delta \|f/|\cdot|\|^2_{L^2(\R^n)}\right\}.
\end{equation}
Since $\|F\|_{L^{2^{*}}(\R^+; d\mu)} \le
|\mathbb{S}^{n-1}|^{-1/2^{*}} \|f\|_{L^{2^{*}}(\R^n)},$ by
H\"{o}lder's inequality, (\ref{Stubbepre}) is implied by Stubbe's
result in \cite{ST}, Theorem 1, namely
\begin{equation}\label{Stubbe}
    \|f\|^2_{L^{2^{*}}(\R^n)} \le
    K(n)[\frac{(n-2)^2}{4} -\delta]^{-\frac{(n-1)}{n}}\left\{\|\grad f\|^2_{L^2(\R^n)} -
    \delta\|\frac{f}{|\cdot|}\|^2_{L^2(\R^n)}\right \}
\end{equation}
with optimal constant
\begin{equation}\label{Stubbe2}
K(n) = [\pi n(n-2)]^{-1}\left (\Gamma(n)/\Gamma(n/2)\right)^{2/n}
[(n-2)^2/4]^{(n-1)/n}.
\end{equation}

\bigskip

We also establish the following local Hardy-Sobolev type
inequalities (see Corollaries 6 and 7): if $f$ is supported in the
annulus $A(1/R,R):= \{\x \in \R^n: 1/R \le |\x| \le R\}$ then
\begin{equation}\label{IntcompactMHS}
    \|r F(r)\|^2_{L^{2^{*}}(\R^+;d\mu)} \le C
   (\ln R)^{2(n-1)/n}\left \{\|Lf \|^2_{L^2(\R^n)}
    -(n^2/4)\|f\|^2_{L^2(\R^n)}\right \};
\end{equation}

\begin{equation}\label{IntcompactHSpenult1}
    \|F\|^2_{L^{2^{*}}(\R^+;d\mu)} \le C
    (\ln R)^{2(n-1)/n}\left \{\|\grad f \|^2_{L^2(\R^n)}
    -\Big  [\frac{n-2}{2}\Big ]^2 \Big \|\frac{f}{|\cdot|}\Big \|^2_{L^2(\R^n)}\right \}.
\end{equation}

The inequality (\ref{IntcompactHSpenult1}) is reminiscent of the
case $s=1$ of (2.6) in \cite{FLS} (proved in section 6.4); this is
also proved in \cite{BVa}. To be specific, it is that if $ f \in
C_0^{\infty}(\Omega)$ and $2 \le q <2^{*},$
\begin{equation}\label{1.7B}
\|f\|^2_{L^q(\R^n)} \le C |\Omega|^{2(1/q - 1/2^{*})}
\left\{\|\nabla f\|^2_{L^2(\R^n)} -\Big [\frac{n-2}{2}\Big ]^2
    \Big \|\frac{f}{|\cdot|}\Big \|^2_{L^2(\R^n)}\right \},
\end{equation}
where $|\Omega|$ denotes the volume of $\Omega.$ It is noted in
\cite{FLS}, Remark 2.4, that, in contrast to
(\ref{IntcompactHSpenult1}), the $q$ in (\ref{1.7B}) must be
strictly less than the critical Sobolev exponent $2^{*} =
2n/(n-2)$ if $\Omega$ includes the origin.

The authors are grateful to Rupert Frank, Elliot Lieb and Robert
Seiringer for some valuable comments.

\section{The Hardy-type inequality (\ref{1.2})}

\begin{Theorem}\label{BE}
Let $n\ge 1$ and $1\le p <\infty.$ Then for all $f\in
C_0^{\infty}(\R^n)$
\begin{equation}\label{BE1}
    \int_{\R^n} |(\x\cdot\nabla)f|^p d\x \ge
    \left(\frac{n}{p}\right)^p \int_{\R^n}|f|^p d\x.
\end{equation}
\end{Theorem}
\begin{proof}
For any differentiable function $V: \R^n \rightarrow \R^n$ we have
\begin{eqnarray}\label{BE3}
\int_{\R^n} \text{div} V |f|^p d\x &=& - p\ Re\int_{\R^n} (V\cdot
\nabla f)
|f|^{p-2} \overline{f} d\x \nonumber \\
&\le& p \left(\int_{\R^n} |V \cdot \nabla f|^p d\x
\right)^{1/p}\left(\int_{\R^n} | f|^p d\x\right)^{(p-1)/p}
\nonumber
\\
& \le & \varepsilon^p \int_{\R^n} |V \cdot \nabla f|^p d\x +(p-1)
\varepsilon^{-p/(p-1)}\int_{\R^n} | f|^p d\x \nonumber \\
\end{eqnarray}
for any $ \varepsilon >0.$ Now choose $V(\x) = \x$ to get
\[
 \int_{\R^n} |(\x \cdot \nabla)f|^p d\x  \ge  K(n,\varepsilon) \int_{\R^n}|f|^p d\x
\]
where
\[
K(n,\varepsilon) = \varepsilon^{-p} \{n-(p-1)
\varepsilon^{-p/(p-1)}\}.
\]
This takes its maximum value $ (n/p)^p$ when
$\varepsilon^{p/(p-1)}= p/n.$ This proves the theorem.
\end{proof}

\begin{Remark}
The inequality (\ref{BE1}) implies (\ref{1.1}) for $1 \le p \le
n.$ For we have from
\begin{equation*}
  \nabla(|\x|f) = \frac{\x}{|\x|}f + |\x| \nabla f
\end{equation*}
that
\begin{eqnarray*}
  \|\grad(|\x| f)\|_{L^p(\R^n)} & \ge & \| |\x||\grad f| \|_{L^p(\R^n)} - \|f\|_{L^p(\R^n)} \\
   &\ge & \|(\x \cdot \grad)f \|_{L^p(\R^n)} - \|f\|_{L^p(\R^n)} \\
   & \ge & \left(\frac{n-p}{p}\right) \|f\|_{L^p(\R^n)}
\end{eqnarray*}
whence (\ref{1.1}) on replacing $f(\x)$ by $f(\x)/|\x|.$
\end{Remark}

\section{Calculation of the semigroup $e^{-tL^{*}L}$}

\begin{Theorem}\label{H}
 Let $L= \x\cdot \nabla, \x=r\omega, r=|\x|$.
 Then the semigroup $e^{-t L^*L}$ is given by
 \beq\label{eq:representation1}
  (e^{-tL^*L}\psi)(\x)
  =
  \frac{e^{-tn^2/4}}{\sqrt{4\pi t}}
  r^{-n/2}
  \int_0^\infty e^{-\frac{(\ln r-\ln s)^2}{4t}} s^{-n/2} \psi (s\omega) \,
  s^{n-1}ds
 \eeq
\end{Theorem}
\begin{proof} Before embarking on the proof, some preliminary
remarks and results might be helpful. The gist of the proof is
that after a change of co-ordinates, $L^{*}L$ is seen to be
related to the Laplacian in $\R,$ and this then yields the result.
The co-ordinate change is determined by the map $\Phi: L^2(\R^n)
\rightarrow L^2(\R \times \mathbb{S}^{n-1})$ defined by
\beq\label{eq:Phi}
  (\Phi\psi)(s,\omega):= e^{sn/2}\psi(e^s \omega)
 \eeq
for $\omega\in \mathbb{S}^{n-1}$ and $s\in\R$. Note that we equip
$\R\times \mathbb{S}^{n-1}$ with the usual one dimensional
Lebesgue measure on $\R$ and the usual surface measure on
$\mathbb{S}^{n-1}$. Thus $\Phi$ preserves the $L^2$ norm. The
inverse of $\Phi$ satisfies $\Phi^{-1}: L^2(\R\times
\mathbb{S}^{n-1})\to L^2(\R^n)$ and is given by
 \beq\label{eq:Phiinverse}
  (\Phi^{-1}\varphi) (\x)= r^{-n/2} \varphi\big(\ln r, \omega \big).
 \eeq

\bigskip

The dilations $U(t): L^2(\R^n)\to L^2(\R^n)$  given by
 \beq\label{eq:dilations}
 U(t)\psi (\x) := e^{tn/2}\psi(e^{t}\x)
 \eeq
form a group of unitary operators with generator $U(t)=e^{iAt},$
where $A$ is given by
 \bdm
  iA\psi = \frac{\partial}{\partial t} U(t) \psi = (\x\cdot \nabla +\frac{n}{2})\psi
  =
   \frac{1}{2}(\x\cdot\nabla +\nabla\cdot \x)\psi.
 \edm
Thus
 \beq\label{eq:generator}
  A = \frac{1}{i}(\x\cdot \nabla +\frac{n}{2}) = -i L -i\frac{n}{2}.
 \eeq
and so
 \bdm
  L= iA -\frac{n}{2},
 \edm
where $A$ is the self-adjoint generator of dilations in
$L^2(\R^n)$. In particular,
 \beq \label{eq:LstarL}
  L^*L = (-iA-\frac{n}{2}) (iA-\frac{n}{2}) = A^2 +\frac{n^2}{4}.
 \eeq
\bigskip

Since
 \beq\label{eq:Phi2}
  (\Phi\psi)(s,\omega) = (U(s)\psi)(\omega)
 \eeq
for $\omega\in \mathbb{S}^{n-1}$ and $s\in\R$,
\bigskip
it follows from the group property of the dilations $U(\cdot)$
that
 \bdm
  (\Phi(U(t)\psi)) (s,\omega)
  =
  (U(s)(U(t)\psi))(\omega)
  =
  (U(s+t)\psi)(\omega) = (\Phi\psi)(s+t,\omega).
 \edm
In particular, in the new co-ordinates given by $\Phi$, the
dilations $U(t)$ act simply as shifts by $t$ and should be
diagonalizable with the help of a Fourier transform! We now
proceed to confirm this prediction.

\bigskip

 Define $M: L^2(\R^n)\to L^2(\R\times S^{n-1})$ by
 \beq\label{eq:M}
  (M\psi)(\tau,\omega):= \frac{1}{\sqrt{2\pi}}\int_\R e^{-is\tau} (\Phi\psi)(s,\omega)\, ds
  ,
 \eeq
 so that $ M= \mathcal{F} \circ \Phi,$ where $ \mathcal{F}$ is the Fourier transform
 on $\R.$ Then
 \begin{align}\label{eq:diagonalU}
 (MU(t)\psi)(\tau,\omega)
 &=
 \frac{1}{\sqrt{2\pi}} \int e^{-is\tau }(\Phi\psi)(s+t,\omega)\, ds  \nonumber \\
 &=
 \frac{e^{it\tau}}{\sqrt{2\pi}}
 \int e^{-is\tau }(\Phi\psi)(s,\omega)\, ds
 =
 e^{it\tau}(M\psi)(\tau,\omega).
\end{align}
The map $M=\cf \circ \Phi$ is the Mellin transformation and has an
explicit representation using the group structure of $\R^+$ under
multiplication: it is the Fourier transform on this group.

\bigskip

The next step is to show that \beq\label{eq:diagonalA}
  (MA\psi)(\tau,\omega) = \tau(M\psi)(\tau,\omega).
 \eeq
 for $\psi$ in the domain $\mathcal{D}(A)$: it follows that $\psi
 \in \mathcal{D}(A)$ if and only if $(\tau,\omega) \mapsto \tau
 (M\psi)(\tau,\omega) \in L^2(\R\times \mathbb{S}^{n-1}).$ To see
 (\ref{eq:diagonalA}) we note that $iAe^{itA} = \partial_t U(t)$
 and so, from (\ref {eq:diagonalU})
 \begin{align*}
  (M iA e^{iAt}\psi)(\tau,\omega)
  &=
  (M \partial_t U(t)\psi)(\tau,\omega)
  =
  \partial_t (MU(t)\psi)(\tau,\omega) \\
  &=
  \partial_t e^{it\tau} (M\psi)(\tau,\omega)
  =
   i\tau e^{it\tau} (M\psi)(\tau,\omega).
 \end{align*}
Setting $t=0$ yields \eqref{eq:diagonalA}.

\bigskip

We are now in a position to complete the proof of the theorem. We
have $e^{-tL^*L}= e^{-tn^2/4}e^{-tA^2}$ and by \eqref{eq:M}
 \beq\label{eq:diagonal-ehoch-tA^2}
  (M e^{-tA^2}\psi)(\tau,\omega) = e^{-t\tau^2} (M\psi)(\tau,\omega).
 \eeq
So
 \bdm
 e^{-tA^2} = M^{-1} e^{-t\tau^2} M .
 \edm
Since $M= \cf\circ \Phi$, we see that
 \beq \label{X}
  e^{-tA^2}=
  \Phi^{-1} \circ \cf^{-1} \big( e^{-t\tau^2} \cf\circ\Phi\big).
 \eeq
Of course,
 \begin{align*}
  \cf^{-1} \big( e^{-t\tau^2} M\psi \big)(\lambda,\omega)
  &=
  \cf^{-1}\big( e^{-t\tau^2} \cf\circ\Phi\big)(\lambda,\omega)\\
  &=
  \frac{1}{2\pi} \int_\R  \int_\R  e^{i\lambda\tau}  e^{-t\tau^2} e^{-is\tau}
  (\Phi\psi) (s,\omega)ds d\tau \\
  &=
  \frac{1}{2\pi} \int_{\R} \Big(\int_{\R}   e^{-t\tau^2 +i(\lambda-s)\tau} d\tau \Big)
  (\Phi\psi)(s,\omega) \, ds
 \end{align*}
The integral in big parentheses is a Gaussian integral which gives
 \bdm
  \int_{\R}   e^{-t\tau^2 +i(\lambda-s)\tau}d\tau
  =
  \sqrt{\frac{\pi}{t}} e^{-\frac{(\lambda-s)^2}{4t}} .
 \edm
Thus
 \bdm
  \cf^{-1} \big( e^{-t\tau^2} M\psi \big)(\lambda,\omega)
  =
  \frac{1}{\sqrt{4\pi t}}
  \int e^{-\frac{(\lambda-s)^2}{4t}} (\Phi\psi)(s,\omega)\, ds
  =:
   \varphi_t(\lambda,\omega)
 \edm
and, with $\x =r\omega,$
 \begin{align*}
  (e^{-tA^2} \psi)(r\omega)&= (\Phi^{-1}\varphi_t)(r\omega) \\
  &= r^{-n/2} \varphi_t(\ln r,\omega) \\
  &=
  \frac{1}{\sqrt{4\pi t}} r^{-n/2}
  \int_\R e^{-\frac{(\ln r-s)^2}{4t}} (\Phi\psi)(s,\omega)\, ds .
 \end{align*}
Since $(\Phi\psi)(s,\omega)= e^{sn/2}\psi(e^s \omega)$, we get
from the change of variables $z= e^s,$
 \begin{align*}
  (e^{-tA^2} \psi)(r\omega)
  &=
  \frac{1}{\sqrt{4\pi t}} r^{-n/2}
  \int_\R e^{-\frac{(\ln r-s)^2}{4t}} (\Phi\psi)(s,\omega)\, ds \\
  &=
  \frac{1}{\sqrt{4\pi t}} r^{-n/2}
  \int_0^\infty e^{-\frac{(\ln r-\ln z)^2}{4t}} z^{\frac{n}{2}-1}\psi(z\omega)
  dz.
 \end{align*}
So
  \begin{align*}
  (e^{-tL^*L}\psi)(r\omega)&= e^{-tn^2/4}(e^{-tA^2} \psi)(r\omega) \\
  &=
  \frac{1}{\sqrt{4\pi t}} r^{-n/2}e^{-tn^2/4}
  \int_0^\infty e^{-\frac{(\ln r-\ln z)^2}{4t}} z^{\frac{n}{2}-1}\psi(z\omega)\, dz \\
  &=
  \frac{1}{\sqrt{4\pi t}} r^{-n/2}e^{-tn^2/4}
  \int_0^\infty e^{-\frac{(\ln r-\ln z)^2}{4t}} z^{-\frac{n}{2}}\psi(z\omega)\,
  z^{n-1}dz
  \end{align*}
which is \eqref{eq:representation1}.

\bigskip

Once it is realised that $A$ is simply multiplication by $\tau$ in
the sense of \eqref{eq:diagonalA}, it is clear that $A$ is the
momentum operator on $\R$, that is, $\Phi A \Phi^{-1}$ is given by
\begin{equation}\label{eq:diagonalA-2}
\Phi A \Phi^{-1} = -i \partial_s \otimes \id_{\mathbb{S}^{n-1}}
\end{equation}
On using this and the functional calculus we get
 \beq\label{eq:diagonalLstarL}
  \Phi L^*L \Phi^{-1} = (\Phi A \Phi^{-1})^2 +\frac{n^2}{4}
  =
  -\partial_s^2 \otimes \id_{S^{n-1}}   +\frac{n^2}{4}.
 \eeq
Thus, $L^*L= - \Phi^{-1} \partial_s^2\otimes \id_{S^{n-1}} \Phi
+\tfrac{n^2}{4}$ and
 \beq\label{eq:diagonal-heatsemigroup-1}
  e^{-tL^*L}
  =
  e^{-tn^2/4} e^{-t\Phi^{-1} \partial_s^2\otimes \id_{S^{n-1}} \Phi}
  =
  e^{-tn^2/4} \Phi^{-1} e^{-t \partial_s^2\otimes \id_{S^{n-1}} } \Phi
 \eeq
which is a convenient way of expressing
\eqref{eq:representation1}.
\end{proof}

\bigskip

On substituting \eqref{eq:Phi} and \eqref{eq:Phiinverse} and
making an obvious change of variables, we obtain from
\eqref{eq:representation1} the following representation for
$e^{-tA^2}$.

\begin{Corollary}\label{H2} Let $P_t$ denote $e^{-tA^2}.$ Then
\begin{equation}\label{C1}
    \Phi P_t \Phi^{-1} \varphi(r,\omega) = \frac{1}{\sqrt{4\pi t}} \int_{\R}
    \exp\{-\frac{1}{4t}(r-s)^2\} \varphi(s\omega) ds.
\end{equation}
\end{Corollary}

\bigskip

\section{The main inequality}
We shall denote the integral mean of a function $f$ on $
\mathbb{S}^{n-1},$ by $ \mathcal{M}(f)(r)$ and when there is no
danger of ambiguity, use the corresponding capital letter; thus
\[
F(r)\equiv \mathcal{M}(f)(r):= |\mathbb{S}^{n-1}|^{-1}
\int_{\mathbb{S}^{n-1}} f(r\omega) d\omega.
\]

We have from (\ref{X})

\begin{equation}\label{Eq1.4B}\begin{array}{rl}
e^{-tL^*L}=& e^{-tn^2/4} e^{-tA^2}\\
e^{-tA^2}=&\Phi^{-1}\circ
\mathcal{F}^{-1}(e^{-t\tau^2}\mathcal{F}\circ \Phi).
\end{array}
\end{equation}
Therefore,
\begin{equation}\label{Eq1.5B}
\Phi[e^{-tA^2}f](r,\omega)=\mathcal{F}^{-1}(e^{-t\tau^2}\mathcal{F}\circ
\Phi)(F) =\frac{1}{\sqrt{2\pi}}\int_{-\infty}^\infty
e^{ir\tau-t\tau^2}\hat{(\Phi F)}(\tau,\omega)d\tau
\end{equation}
in which $\hat{g}:=\mathcal{F}(g)$. However, the representation we
use in our analysis is that given by (\ref{C1}), with now $\Phi f=
g,$
\begin{equation*}
    \Phi P_t \Phi^{-1} g(r,\omega) = \frac{1}{\sqrt{4\pi t}} \int_{\R}
    \exp\{-\frac{1}{4t}(r-s)^2\} g(s\omega) ds,
\end{equation*}
where $P_t :=e^{-tA^2}.$

Define $B^\alpha$ to be the space of all tempered distributions
$g$ on $\R\times \mathbb{S}^{n-1}$ for which the norm
\begin{equation}\label{Eq2.1} \|
g\|_{B^\alpha}:=\underset{t>0}{\sup}\{ t^{-\alpha/2}\| \Phi
e^{-tA^2}\Phi^{-1} |G|\|_{L^\infty(\R)}\}<\infty.
\end{equation}

\begin{Theorem}\label{Main}
Let $ 1\le p <q<\infty$ and suppose that $g$ is such that $\Phi A
\Phi^{-1} g \equiv -i(\partial/\partial r)g \in L^p(\R\times
\mathbb{S}^{n-1})$ and $g \in B^{\theta/(\theta-1)}, \theta =p/q$.
Then there exists a positive constant C, depending on $p$ and $q$
such that
\begin{equation}\label{Eq2.2B}
\|G\|_{L^q(\R)}\le C \|(\partial/\partial r)g
\|^{\theta}_{L^p(\R\times
\mathbb{S}^{n-1})}\|g\|_{B^{\theta/(\theta -1)}}^{1-\theta}.
\end{equation}

\end{Theorem}

\begin{Remark}\label{weakineq}
An intermediate result in the proof is
\[
\|G\|_{L^{q,\infty}(\R)} \le 2^{\theta +1}
\pi^{-\theta/2}|\mathbb{S}^{n-1}|^{-1}\|(\partial/\partial r)g
\|^{\theta}_{L^p(\R\times
\mathbb{S}^{n-1})}\|g\|_{B^{\theta/(\theta -1)}}^{1-\theta},
\]
where $L^{q,\infty}(\R)$ is the weak-$L^q$ space with norm
\[
\|G\|_{L^{q,\infty}(\R)}:= \left \{ \sup_{u>0} [u^q \lambda(|G|
\ge u)]\right \}^{1/q},
\]
where $ \lambda(|G| \ge u)$ denotes the Lebesgue measure of the
set $\{r \in \R: |G(r)| \ge u\}.$
\end{Remark}

\begin{Remark}\label{domain}
Note that the supposition in the theorem implies that
$f=\Phi^{-1}g \in \mathcal{D}(A),$ the domain of the operator $A$
acting in $L^2(\R^n).$
\end{Remark}

To prove the theorem we first need some preliminary results on
$P_t :=e^{-tA^2}.$

\begin{Lemma}\label{Lem4.1}
For all $t>0$
\begin{equation}\label{Eq4.7}
\|\Phi P_t \Phi^{-1}G]\|_{L^\infty(\R)}\le C t^{-1/2p}
\|G\|_{L^p(\R)},
\end{equation}
where $C \le (4\pi)^{-1/2p} (p')^{-1/2p'}.$

\end{Lemma}
\begin{proof}
From (\ref{C1}) we have by H\"{o}lder's inequality that
\begin{eqnarray}\label{Eq1.6BB}
\left|\Phi P_t \Phi^{-1} G(r)\right|&\le& \frac{1}{\sqrt{4\pi
t}}\left(\int_{\R} e^{-\frac{p'}{4t}(r-s)^2}ds\right)^\frac{1}{p'}
\left(\int_{\R} |G(s)|^pds\right)^{1/p} \nonumber \\
&\le & C t^{-1/2p} \|G\|_{L^p(\R)}
\end{eqnarray}
with the indicated constant.
\end{proof}

\begin{Lemma}\label{Lem4.2}
For all $t>0$
$$
\| \Phi A P_t \Phi^{-1}G\|_{L^p(\R)}\le (\pi t)^{-1/2}\|
G\|_{L^p(\R)}
$$
and similarly
\[
\| \Phi A P_t \Phi^{-1}g\|_{L^p(\R\times \mathbb{S}^{n-1})}\le
(\pi t)^{-1/2}\| g\|_{L^p(\R\times \mathbb{S}^{n-1})} \
\]
\end{Lemma}
\begin{proof}
From (\ref{C1}) we have
\begin{equation*}
\frac{d}{dr} \left\{\Phi  P_t \Phi^{-1}G(r)\right \} =
\frac{1}{\sqrt{4\pi t}}\int_{\R}\frac{(s-r)}{2t} \exp
\left(-\frac{1}{4t}[s-r]^2\right) G(s) ds
\end{equation*}
and hence by Young's inequality for convolutions (see \cite{EE},
Theorem V.1.2)
\begin{align*}
   &\|\frac{d}{dr} \left\{\Phi  P_t \Phi^{-1}G(r)\right \}\|_{L^p(\R)} \\
   & \le \left\{\frac{1}{\sqrt{16 \pi t^3}}\int_{\R}|z|\exp
\left(-\frac{1}{4t}z^2\right)dz \right \} \|G\|_{L^p(\R)}\\
& =  (\pi t)^{-1/2}\|G\|_{L^p(\R)}.
\end{align*}
The lemma follows since
\begin{eqnarray}\label{Z} \frac{d}{d
r}\left(\Phi H\right)
&=&\Phi[\frac{n}{2}H+LH]\nonumber \\
&=&i\Phi AH.
\end{eqnarray}
\end{proof}

We are now ready to prove our Theorem \ref{Main}. Note that the
assertion $\Phi A\Phi^{-1} g \equiv -i(\partial/\partial r)g $
follows from (\ref{Z})(see also (\ref{eq:diagonalA-2})). Our proof
is inspired by that of Theorem 1 in \cite{L}.

\begin{proof}

{\bf{Step 1}}

By homogeneity we may assume that $\|
g\|_{B^{\theta/(\theta-1)}}\le 1$, so that for all $r\in\R $ and
$t
>0$
$$
\Phi e^{-tA^2}\Phi^{-1}|G(r)|\le t^{\theta/2(\theta-1)}.
$$
For all $u>0$ define $t_u:=u^{2(\theta-1)/\theta}$ so that
\begin{equation}\label{Eq2.3B}
\Phi e^{-t_uA^2}\Phi^{-1}|G(r)|\le u.
\end{equation}
Let $\lambda$ denote Lebesgue measure on $\R$. With
$P_t:=e^{-tA^2}$,
\begin{eqnarray}\label{Eq2.4B}
u^q\lambda(|G|\ge 2u) & \le & u^q  \lambda(|G(r)-\Phi
P_{t_u}\Phi^{-1}G(r)| \ge u) \nonumber \\& \le & u^{q-p}\int_{\R}
|G(r)-\Phi P_{t_u}\Phi^{-1}G(r)|^pdr \nonumber \\
& = & u^{q-p}
\int_{\R}\left|\frac{1}{|\mathbb{S}^{n-1}|}\int_{\mathbb{S}^{n-1}}[g(r\omega)-\Phi
P_{t_u}\Phi^{-1}g(r\omega)]d\omega\right|^p dr \nonumber \\
& \le & u^{q-p}\frac{1}{|\mathbb{S}^{n-1}|} \|g -\Phi
P_{t_u}\Phi^{-1}g\|^p_{L^p(\R \times \mathbb{S}^{n-1})}.
\end{eqnarray}

Since $f:= \Phi^{-1}g$ is assumed to lie in $\mathcal{D}(A)$, the
domain of $A$, we have
$$
\frac{\partial}{\partial t}P_tf =A^2P_tf,\q P_0f=f,\ \
$$
and consequently
$$
(P_tf-f)(t)=\int_0^t A^2P_s f ds.
$$
Set $k:= \Phi^{-1} h$ where $h \in C_0^{\infty}(\R \times
\mathbb{S}^{n-1}).$ Then $k \in C_0^{\infty}(\R^n \setminus\{0\})$
and hence lies in $ \mathcal{D}(A).$ We therefore have with
$\x=(r\omega)$
 \begin{align*}
&\int_{\R\times \mathbb{S}^{n-1}} h(r\omega)(\Phi P_t
\Phi^{-1}g-g)(r\omega)drd\omega =
\int_{\R^n} k(\x) \left(P_t f(\x) -f(\x)\right)d\x \\
&= \int_0^t\int_{\R^n} k(\x)A^2P_s f(\x) d\x ds\\
&=\int_0^t\int_{\R^n} [AP_sk](\x)[Af](\x)d\x ds\\
&= \int_0^t\int_{\R\times \mathbb{S}^{n-1}} [\Phi AP_s
\Phi^{-1}h](r\omega)[\Phi
A\Phi^{-1}g](r\omega)dr d \omega ds\\
&\le  \|\Phi A \Phi^{-1}g\|_{L^p(\R\times
\mathbb{S}^{n-1})}\int_0^t\|\Phi AP_s \Phi^{-1}
h\|_{L^{p'}(\R\times
\mathbb{S}^{n-1})}ds\\
&\le  2\pi^{-\frac12} t^{\frac12} \|\Phi A
\Phi^{-1}g\|_{L^p(\R\times \mathbb{S}^{n-1})}\|
h\|_{L^{p'}(\R\times \mathbb{S}^{n-1})}
\end{align*}
by Lemma~\ref{Lem4.2}. Since $C_0^{\infty}(\R\times
\mathbb{S}^{n-1})$ is dense in $L^{p'}(\R\times \mathbb{S}^{n-1})$
we obtain the pseudo-Poincar\'{e} inequality (see \cite{S})
\begin{equation}\label{Eq2.5B}
\|\Phi P_t \Phi^{-1}g-g\|_{L^p(\R\times \mathbb{S}^{n-1})}\le
2\pi^{-\frac12}t^{\frac12} \|\Phi A\Phi^{-1}g\|_{L^p(\R\times
\mathbb{S}^{n-1})}.
\end{equation}
Thus, in (\ref{Eq2.4B}),
\begin{eqnarray}\label{Eq2.6}
u^q\lambda(|G|\ge 2u)&\le & 2^p\pi^{-\frac{p}{2}} u^{q-p}
t_u^{p/2}|\mathbb{S}^{n-1}|^{-1}
\|\Phi A \Phi^{-1}g\|^p_{L^p(\R\times \mathbb{S}^{n-1})}\nonumber \\
&= & 2^p \pi^{-\frac{p}{2}}|\mathbb{S}^{n-1}|^{-1} \|\Phi A
\Phi^{-1}g\|^p_{L^p(\R\times \mathbb{S}^{n-1})},
\end{eqnarray}
whence
\begin{equation}\label{Eq2.7}
\|G\|_{L^{q,\infty}(\R)}\le   2^{\theta
+1}\pi^{-\theta/2}|\mathbb{S}^{n-1}|^{-1}\|\Phi A
\Phi^{-1}g\|^\theta_{L^p(\R\times \mathbb{S}^{n-1})},
\end{equation}
where $L^{q,\infty}$ denotes the weak $L^q$ norm.

{\bf{Step 2}}

In this step we show that the $L^{q,\infty}$ norm in (\ref{Eq2.7})
can be replaced by the $L^q$ norm if we assume that $G \in
L^q(\R).$ We may, and shall hereafter in the proof, assume that
our functions $G$ are real-valued. Following Ledoux in \cite{L},
we write
\begin{equation}\label{1}
    5^{-q}\|G\|^q_{L^q(\R)} = \int_0^{\infty}\lambda(|G| \ge 5u\})
    du^q
\end{equation}
and for $u>0$ define $G_u$ by
\begin{equation}\label{2}
    G_u =(G-u)^+ \wedge ((c-1)u) + (G+u)^-\vee(-(c-1)u)
\end{equation}
where $c\ge5,$ and $\wedge, \vee $ denote the minimum and maximum
respectively. It follows that for $u\le |G| \le cu$
\begin{equation}\label{3}
    \frac{d}{dr} G_u = \frac{d}{dr} G
\end{equation}
and is zero otherwise. Also,
\begin{equation}\label{4}
    |G| \ge 5u \Longrightarrow |G_u| \ge 4u
\end{equation}
and hence
\begin{equation}\label{5}
    \int_0^{\infty} \lambda(|G| \ge 5u) du^q \le  \int_0^{\infty} \lambda(|G_u| \ge 4u)
    du^q.
\end{equation}
We continue to assume that $\|g\|_{B^{\theta/(\theta-1)}}\le 1$
and have $t_u = u^{2(\theta-1)/\theta}, \theta = p/q.$ We have
\begin{eqnarray}\label{5}
  |G_u|  &\le & |G_u-\Phi P_{t_u}\Phi^{-1}G_u| + |\Phi P_{t_u}\Phi^{-1}[G_u-G]| +|\Phi P_{t_u}\Phi^{-1}G| \nonumber \\
   &\le & |G_u-\Phi P_{t_u}\Phi^{-1}G_u| + \Phi P_{t_u}\Phi^{-1}|G_u-G| +u
\end{eqnarray}
since $ |\Phi P_{t_u}\Phi^{-1} G|\le \Phi P_{t_u}\Phi^{-1} |G| \le
u.$ Thus $|G_u| \ge 4u$ implies that
\begin{equation}\label{6}
|G_u-\Phi P_{t_u}\Phi^{-1}G_u| + \Phi P_{t_u}\Phi^{-1}|G_u-G| \ge
3u .
\end{equation}
This in turn implies that the set $ \{r:|G_u| \ge 4u\}$ is
contained in $\{r:|G_u- \Phi P_{t_u}\Phi^{-1}G_u| \ge u \}\bigcup
\{r:\Phi P_{t_u}\Phi^{-1}|G_u-G| \ge 2u \}$. It follows that
\begin{eqnarray}\label{7}
  \int_0^{\infty} \lambda(|G_u| \ge 4u) du^q & \le & \int_0^{\infty} \lambda( |G_u-\Phi P_{t_u}\Phi^{-1}G_u| \ge
u) du^q \nonumber \\
   &+& \int_0^{\infty} \lambda(\Phi P_{t_u}\Phi^{-1}|G_u-G| \ge 2u)
   du^q.
   \nonumber \\
\end{eqnarray}

From the pseudo-Poincar\'{e} inequality (\ref{Eq2.5B}) we have,
with $C= 2\pi^{-1/2},$
\begin{equation}\label{9}
  \|G_u -\Phi P_{t_u}\Phi^{-1} G_u\|_{L^p(\R)}
   \le Ct_u^{1/2} \|\Phi A \Phi^{-1}G_u\|_{L^p(\R)}
\end{equation}
and hence, on using (\ref{Z}), (4.15) and (4.21), and recalling
that $ t_u=u^{2(\theta-1)/\theta},$ so that $u^{-p} t_u^{p/2} =
u^{-q},$
\begin{eqnarray}\label{10}
   \lambda( |G_u-\Phi P_{t_u}\Phi^{-1}G_u| \ge
u)& \le & u^{-p} \int _0^{\infty} |G_u-\Phi P_{t_u}\Phi^{-1}G_u|^p dr \nonumber \\
   &\le & Cu^{-p}t_u^{p/2} \|\Phi A \Phi^{-1}G_u \|^p_{L^p(\R)} \nonumber \\
   &=& C u^{-q} \|\frac{d}{dr}G_u \|^p_{L^p(\R)}\nonumber  \\
      & = & Cu^{-q} \int_{u<|G|<cu}|\frac{d}{dr}G|^p dr
   \nonumber \\
   & = & C u^{-q} \int_{u<|G|<cu} |\Phi A \Phi^{-1}G|^p
   dr.\nonumber \\
\end{eqnarray}
Hence
\begin{align}\label{11}
 & \int_0^{\infty} \lambda( |G_u- \Phi P_{t_u}\Phi^{-1}G_u| \ge
u)du^q \nonumber \\
& \le  C \int_0^{\infty} \left\{ u^{-q} \int_{u<|G|<cu} |\Phi A \Phi^{-1}G|^p dr\right \} d u^q \nonumber  \\
   &= C \int_{\R} |\Phi A \Phi^{-1}G (r)|^p \left \{\int_{|G|/c}^{|G|} u^{-q} d u^q \right \}dr\nonumber  \\
   &= C q \ln c \|\Phi A \Phi^{-1}G\|^p _{L^p(\R)} \nonumber \\
   & \le C q \ln c \frac{1}{|\mathbb{S}^{n-1}|}\|(\partial/\partial r)g\|^p _{L^p(\R\times \mathbb{S}^{n-1})}
\end{align}
by (4.7).

\bigskip

Next we consider $ \lambda(\Phi P_{t_u}\Phi^{-1}|G_u -G| \ge 2u).$
First, we claim that
\begin{equation}\label{12}
    \Phi P_{t_u}\Phi^{-1}|G_u -G| \le u + \Phi P_{t_u}\Phi^{-1}|G| \chi_{\{|G| \ge
    cu\}},
\end{equation}
where $\chi_I$ denotes the characteristic function of the set $I.$
We have from (4.14)
\begin{eqnarray}\label{14}
    |G_u -G|&\le &|G_u -G|\chi_{\{|G| \le cu\}} + |G_u -G|\chi_{\{|G| \ge
    cu\}}\nonumber \\
    &\le & u + |G_u -G|\chi_{\{|G| \ge
    cu\}}.
\end{eqnarray}
 Hence, from (\ref{C1}),
\begin{align}\label{15}
  &\Phi P_{t_u}\Phi^{-1}| G_u -G|
    \le  \frac{u}{\sqrt{4\pi t_u}} \int_{\R}\exp\{-\frac{1}{4t_u}(r-s)^2 \}
   ds \nonumber \\
   & +  \frac{1}{\sqrt{4\pi t_u}}\int_{\R}\exp\{-\frac{1}{4t_u}(r-s)^2 \}|G- G_u|\chi_{\{|G| \ge cu\}}
   ds \nonumber \\
   & =  u + \frac{1}{\sqrt{4\pi t_u}}\int_{\R}\exp\{-\frac{1}{4t_u}(r-s)^2 \}|G- G_u|\chi_{\{|G| \ge cu\}}
   ds .\nonumber \\
\end{align}
For $|G| \ge cu,$ we have from the construction of $G_u$ in (4.14)
that
\begin{equation}\label{16}
    |G-G_u| \le |G|
\end{equation}
and hence on substituting in (\ref{15}) we get
\begin{eqnarray}\label{16}
  \Phi P_{t_u}\Phi^{-1}|G_u -G|& \le & u + \frac{1}{\sqrt{4\pi t_u}}\int_{\R}\exp\{-\frac{1}{4t_u}(r-s)^2 \}|G|
  \chi_{\{|G| \ge cu\}}
   ds \nonumber \\
   & = & u + \Phi P_{t_u}\Phi^{-1}|G|\chi_{\{|G| \ge cu\}},
\end{eqnarray}
as claimed in (4.24). This gives
\begin{align}\label{17}
  &\int_0^{\infty} \lambda(\Phi P_{t_u}\Phi^{-1}|G_u-G| \ge 2u) du^q
  \nonumber \\
 &\le  \int_0^{\infty}\lambda(\Phi P_{t_u}\Phi^{-1}|G|\chi_{\{|G| \ge cu\}}\ge u) du^q \nonumber \\
   &\le  \int_0^{\infty} u^{-1} \left(\int_{\R}\Phi
   P_{t_u}\Phi^{-1}|G|\chi_{\{|G| \ge cu\}} dr
   \right) du^q \nonumber \\
   & =  \int_0^{\infty} \frac{1}{\sqrt{4\pi
   t_u}}\int_{\R}
    \left[\int_0^{\infty}\exp\{-\frac{1}{4t_u}(r-s)^2\} |G|
    \chi_{\{|G| \ge cu\}}ds\right] dr \frac{du^q}{u} \nonumber \\
   & \le  \int_0^{\infty}u^{-1}\int_0^{\infty}|G|\chi_{\{|G| \ge
   cu\}}ds du^q \nonumber \\
    &=  q \int_0^{\infty} |G| \left( \int_0^{|G|/c}
   u^{q-2}du \right)
   ds \nonumber \\
   & =  \frac{q}{(q-1) c^{q-1}}\|G\|^q _{L^q(\R)}.
\end{align}
We have therefore shown that \[ 5^{-q}\|G\|^q_{L^q(\R)} \le C q
\ln c\|(\partial/\partial r)g\|^p_ {L^p(\R\times
\mathbb{S}^{n-1})} + \frac{q}{(q-1)c^{q-1}}\|G\|^q_{L^q(\R)}
\]
which on choosing $c$ large enough yields (4.4) under the
additional assumption $G\in L^q(\R).$

{\bf{Step 3}}

The final step is to remove the assumption $ G\in L^q(\R)$ in Step
2. We again follow Ledoux's approach and define
\[
N_{\varepsilon}(G) = \int_{\varepsilon}^{1/\varepsilon}
\lambda(|G|\ge 5u) d(u^q) < \infty.
\]
From (4.17), (4.20), (4.23) and (4.29) it is seen that
\begin{equation}\label{1}
N_{\varepsilon}(G)\le C q \ln c \|(\partial/\partial r)g\|^p_
{L^p(\R\times \mathbb{S}^{n-1})}
+\int_{\varepsilon}^{1/\varepsilon} \frac{1}{u}\left(\int
|G|\chi_{\{|G|>cu\}}d \lambda\right)d(u^q).
\end{equation}
We shall use the fact that
\begin{equation}\label{2}
\int |G|\chi_{\{|G|>cu\}}d \lambda = -\int_{cu}^{\infty}\alpha d
\lambda(\alpha)
\end{equation}
where
\[
\lambda(\alpha) := \lambda\{x: |G(x)| > \alpha\}.
\]
On integration by parts, we have for all $\Lambda > cu,$ that
\begin{eqnarray*}
-\int_{cu}^{\Lambda} \alpha d \lambda(\alpha) &=& -\left[\alpha
\lambda(\alpha)\right]_{cu}^{\Lambda} +
\int_{cu}^{\infty}\lambda(\alpha) d\alpha \\
& \le & cu \lambda(cu) + \int_{cu}^{\infty} \lambda(\alpha)d
\alpha
\end{eqnarray*}
and hence
\begin{equation}\label{3}
    \int |G|\chi_{\{|G|>cu\}}d \lambda \le cu \lambda(cu) + \int_{cu}^{\infty} \lambda(\alpha)d
\alpha
\end{equation}
 From this we infer that
\begin{eqnarray}\label{4}
 I &:=& \int_{\varepsilon}^{1/\varepsilon} \frac{1}{u}\left(\int
|G|\chi_{\{|f|>cu\}}d \lambda\right)d(u^q)\nonumber \\
 & \le & c\int_{\varepsilon}^{1/\varepsilon}\lambda(cu) d(u^q)
 + \int_{\varepsilon}^{1/\varepsilon}\left( \int_{cu}^{\infty}\lambda(\alpha) d\alpha\right)qu^{q-2}du \nonumber \\
   &=& c\int_{\varepsilon}^{1/\varepsilon}\lambda(cu) d(u^q) + I_1
\end{eqnarray}
say. We now apply Fubini's Theorem to $I_1.$
\begin{eqnarray}\label{5}
  I_1 &=& \int_{\alpha = c\varepsilon}^{c/\varepsilon}\lambda(\alpha) d\alpha
  \int_{u=\varepsilon}^{\alpha/c} qu^{q-2}du \nonumber \\
   &+& \int_{\alpha = c/\varepsilon}^{\infty}\lambda(\alpha) d\alpha
  \int_{u=\varepsilon}^{1/\varepsilon} qu^{q-2}du \nonumber \\
  &=& c \int_{t = \varepsilon}^{1/\varepsilon}\lambda(ct) dt  \left[
  \frac{q}{(q-1)}u^{q-1}\right]_{\varepsilon}^t + c \int_{t = 1/\varepsilon}^{\infty}\lambda(ct) dt  \left[
  \frac{q}{(q-1)}u^{q-1}\right]_{\varepsilon}^{1/\varepsilon}
  \nonumber \\
  &\le & \frac{cq}{(q-1)}\int_{ \varepsilon}^{1/\varepsilon}t^{q-1}\lambda(ct)
  dt+ \frac{cq}{(q-1)}\frac{1}{\varepsilon^{q-1}}\int_{1/\varepsilon}^{\infty}\lambda(ct)
  dt\nonumber \\
  &=& \frac{c}{(q-1)}\int_{ \varepsilon}^{1/\varepsilon}\lambda(ct)
  d(t^{q})+ \frac{cq}{(q-1)}\frac{1}{\varepsilon^{q-1}}\int_{ 1/\varepsilon}^{\infty}\lambda(ct)
  dt.
\end{eqnarray}
It follows from (\ref{4}) and (\ref{5}) that
\begin{equation}\label{6}
    I \le \frac{cq}{(q-1)}\int_{ \varepsilon}^{1/\varepsilon}\lambda(ct)
  d(t^{q})+ \frac{cq}{(q-1)}\frac{1}{\varepsilon^{q-1}}\int_{ 1/\varepsilon}^{\infty}\lambda(ct)
  dt.
\end{equation}
On setting $ t= (c/5)u, \varepsilon = (5/c)\tilde{\varepsilon}$ we
have
\begin{eqnarray}\label{7}
  \frac{cq}{(q-1)}\int_{ \varepsilon}^{1/\varepsilon}\lambda(ct)
  d(t^{q}) &=& \frac{q}{(q-1)}\frac{5^q}{c^{q-1}}N_{ \tilde{\varepsilon}}(G) \nonumber \\
   &\le &  \frac{q}{(q-1)}\frac{5^q}{c^{q-1}}N_{ \varepsilon}(G)
\end{eqnarray}
since $ \tilde{\varepsilon} \ge \varepsilon.$ We also have in
(\ref{6})
\begin{eqnarray*}
  \int_{1/\varepsilon}^{\infty}\lambda(|G|>cu)du  &=&
  \int_{1/\varepsilon}^{\infty}(cu)^q\lambda(|G|>cu)(cu)^{-q}du  \\
   &\le & \frac{1}{c^q}\|G \|^q_{q,\infty}
   \int_{1/\varepsilon}^{\infty} u^{-q} du  \\
   &=& \frac{\varepsilon^{q-1}}{c^q(q-1)}\|G\|^q_{L^{q,\infty}(\R)}
\end{eqnarray*}
and so
\begin{equation}\label{8}
    \frac{cq}{(q-1) \varepsilon^{q-1}} \int_{1/\varepsilon}^{\infty}\lambda(|G|>cu)du
    \le \frac{q}{(q-1)^2c^{q-1}}\|G\|^q_{L^{q,\infty}(\R)}.
\end{equation}
We therefore have from (\ref{1})
\begin{eqnarray}\label{9}
  N_{\varepsilon}(G) &\le & C q \ln c \|(\partial/\partial r)g\|^p_ {L^p(\R\times
\mathbb{S}^{n-1})}
  + \frac{q}{(q-1)}\frac{5^q}{c^{q-1}}N_{ \varepsilon}(G)\nonumber  \\
    &+& \frac{q}{(q-1)^2c^{q-1}}\|G\|^q_{L^{q,\infty}(\R)}.
\end{eqnarray}
On choosing $c$ large enough it follows that $\sup_{\varepsilon
>0}N_{\varepsilon}(G) < \infty $ and so $G \in L^q(\R).$
The proof is therefore complete.
\end{proof}

The theorem has two natural corollaries featuring the Hardy-type
inequality (2.1), the first an inequality of Sobolev type , and
the second of Gagliardo-Nirenberg type.

\begin{Corollary}\label{Sob}
Let $ p^{*} := np/(n-p), 1\le p <n, $ and suppose $g,
(\partial/\partial r)g \in L^p(\R\times \mathbb{S}^{n-1}).$ Then
\begin{equation}\label{Sob1}
    \|G\|_{L^{p^{*}}(\R)} \le C \|(\partial/\partial r)g\|^{1/n}_{L^p(\R\times \mathbb{S}^{n-1})}
    \|g\|^{(n-1)/n}_{L^p(\R\times \mathbb{S}^{n-1})}.
\end{equation}
If $G$ is supported in $[-\Lambda, \Lambda],$ then
\begin{equation}\label{compactSob1}
    \|G\|_{L^{p^{*}}(\R)} \le C \Lambda^{(n-1)/n}\|(\partial/\partial r)g\|_{L^p(\R\times
    \mathbb{S}^{n-1})}.
\end{equation}
\end{Corollary}
\begin{proof}
From Lemma \ref{Lem4.1}
\begin{eqnarray*}
 t^{-\theta/2(\theta
-1)} \|\Phi P_t \Phi^{-1} |G|\|_{L^{\infty}(\R)} &\le & C
t^{-\theta/2(\theta -1) -1/2p} \|G\|_{L^p(\R)} \\
&\le & C \|G\|_{L^p(\R)}\\
& \le & C \|g\|_{L^p(\R\times \mathbb{S}^{n-1}})
\end{eqnarray*}
if $ \theta = p/q, q=p(p+1).$ Hence from Theorem \ref{Main}
\begin{equation}\label{Sob2}
    \|G\|_{L^{p(p+1)}(\R)} \le C \|(\partial/\partial r)
    g\|_{L^p(\R \times \mathbb{S}^{n-1})}^{1/(p+1)}\|g\|_{L^p(\R\times \mathbb{S}^{n-1})}^{p/(p+1)}.
\end{equation}
Thus $G \in L^{p(p+1)}(\R) \cap L^p(\R), $ and since
\[ \frac{np}{(n-p)} =
\frac{p(p+1)}{(n-p)}+ \frac{p(n-p-1)}{(n-p)} \] we have by
H\"{o}lder's inequality,
\begin{eqnarray*}
\int_{\R}|G|^{p^{*}}dr &\le&
\left(\int_{\R}|G|^{p(p+1)}dr\right)^{1/(n-p)}\left(\int_{\R}|G|^pdr\right)^{(n-p-1)/(n-p)}
\\
& \le &\left(\int_{\R}|G|^{p(p+1)}dr\right)^{1/(n-p)}
\left(\frac{1}{|\mathbb{S}^{n-1}|}\int_{\R\times
\mathbb{S}^{n-1}}|g|^pdr d\omega \right)^{(n-p-1)/(n-p)}.
\end{eqnarray*}
Hence, from (4.41),
\begin{eqnarray*}
  \|G\|_{L^{p^{*}}(\R)} &\le& C \|G\|^{(p+1)/n}_{L^{p(p+1)}(\R)}\|g\|^{(n-p-1)/n}_{L^p(\R\times
\mathbb{S}^{n-1})} \\
   &\le & C \|(\partial/\partial r)g \|_{L^p(\R\times
\mathbb{S}^{n-1})}^{1/n}
   \|g\|_{L^p(\R\times
\mathbb{S}^{n-1})}^{(n-1)/n}.
\end{eqnarray*}
The inequality (\ref{compactSob1}) follows on using H\"{o}lder's
inequality to give
\[
\|G\|_{L^p(\R)} \le \|G\|_{L^{p^{*}}(\R)}
(2\Lambda)^{(1/p)-(1/p^{*})}
\]
and then substituting in
\begin{equation*}
    \|G\|_{L^{p(p+1)}(\R)} \le C \|(\partial/\partial r)
    g\|_{L^p(\R \times \mathbb{S}^{n-1})}^{1/(p+1)}\|G\|_{L^p(\R)}^{p/(p+1)}
\end{equation*}
which is proved in the course of establishing (4.41).
\end{proof}

\begin{Corollary}\label{GagNir}
Let $1\le p<q<\infty, m=(q/p)-1,$ and suppose that
$(\partial/\partial r)g \in L^p(\R\times \mathbb{S}^{n-1}), g\in
L^m(\R\times \mathbb{S}^{n-1}).$ Then
\begin{equation}\label{GNineq}
    \|G\|_{L^q(\R)} \le C \|(\partial/\partial r)g\|_{L^p(\R\times
\mathbb{S}^{n-1})}^{p/q}
    \|g\|_{L^m(\R\times
\mathbb{S}^{n-1})}^{1-p/q}.
\end{equation}
\end{Corollary}
\begin{proof}
From Lemma \ref{Lem4.1}, with $\theta =p/q$ and $m=q/p-1,$
\begin{eqnarray*}
 t^{-\theta/2(\theta
-1)} \|\Phi P_t \Phi^{-1} |G|\|_{L^{\infty}(\R)} &\le & C
t^{-\theta/2(\theta -1) -1/2m} \|G\|_{L^m(\R)} \\
&\le & C \|g\|_{L^m(\R\times \mathbb{S}^{n-1})}
\end{eqnarray*}
and this yields (\ref{GNineq}).
\end{proof}

The cases $p=2$ of Corollaries \ref{Sob} and \ref{GagNir} are of
special interest.

\begin{Corollary} \label{Sob2}
Let $f$ be such that $f, Lf \in L^2(\R^n),$ where $ L= \x \cdot
\grad.$ Then for $n>2,$
\begin{eqnarray}\label{Sob3}
    \|r F(r)\|^2_{L^{2^{*}}(\R^+;d\mu)} &\le &C \left\{ \|Lf\|^2_{L^2(\R^n)}
    -\frac{n^2}{4}\|f\|^2_{L^2(\R^n)}\right \}^{1/n}
    \nonumber \\
    & \times&
    \|f\|_{L^2(\R^n)}^{2(1-1/n)},
\end{eqnarray}
where $F = \mathcal{M}(f),2^{*} = 2n/(n-2)$ and $d\mu =
r^{n-1}dr.$
\end{Corollary}
\begin{proof}
On using the facts that $\Phi: L^2(\R^n) \rightarrow L^2(\R\times
\mathbb{S}^{n-1})$ is an isometry and, with $g:= \Phi f,$
\begin{eqnarray*}
 \|(\partial/\partial r)g \|^2_{L^2(\R \times \mathbb{S}^{n-1})} &=&
 \|\Phi A \Phi^{-1}g \|^2_{L^2(\R \times \mathbb{S}^{n-1})} \\
 & = & \|Af\|^2_ {L^2(\R^n)} \\
   &=& \|Lf\|^2_{L^2(\R^n)} - \frac{n^2}{4}
   \|f\|^2_{L^2(\R^n)}
\end{eqnarray*} since $A^2 = L^{*}L - (n^2/4)$ from (3.6), it follows from
(\ref{Sob1}) that
\begin{eqnarray*}
\|\mathcal{M}(\Phi f)\|^2_{L^{2^{*}}(\R)} & \le & C \left\{
\|Lf\|^2_{L^2(\R^n)}
    -\frac{n^2}{4}\|f\|^2_{L^2(\R^n)}\right \}^{1/n} \\
    & \times&
    \|f\|_{L^2(\R^n)}^{2(1-1/n)}.
\end{eqnarray*}
The corollary follows since
\[
\|\mathcal{M}(\Phi f)\|_{L^{2^{*}}(\R)} =
\|rF(r)\|_{L^{2^{*}}(\R^+;d\mu)}.
\]

\end{proof}

\begin{Corollary}\label{Hardy}
Let $ h, \grad h \in L^2(\R^n), n\ge 3.$ Then there exists a
positive constant $C$ depending only on $n$ such that
\begin{eqnarray}\label{NewHardy}
\|\mathcal{M}(h)\|^2_{L^{2^{*}}(\R^+;d\mu)}&\le & C \big \{\|\grad
h \|^2_{L^2(\R^n)} - \big(\frac{n-2}{2}\big)^2
\|h/|\cdot|\|^2_{L^2(\R^n)}\big \}^{1/n} \nonumber \\
& \times & \big \{ \|h/|\cdot|\|^2_{L^2(\R^n)}\big \}^{1-1/n}.
\end{eqnarray}
Hence, for any $\varepsilon >0 ,$
\begin{eqnarray}\label{NewHardy2}
 \varepsilon^{1-1/n}
\|\mathcal{M}(h)\|^2_{L^{2^{*}}(\R^+;d\mu)}&\le & C
\{\|\grad h \|^2_{L^2(\R^n)} \nonumber \\
&-& [\big(\frac{n-2}{2}\big)^2 - \varepsilon  ]
\|h/|\cdot|\|^2_{L^2(\R^n)}  \}.
\end{eqnarray}
\end{Corollary}
\begin{proof}
Since $n\ge 3,$ we have that $ f:= h/|\cdot| \in L^2(\R^n).$ We
claim that $Lf \in L^2(\R^n).$ For
\begin{eqnarray*}
  |\nabla(|\x|f)|^2 &=& \left|\frac{\x}{|\x|}f + |\x| \nabla f\right|^2 \\
   &=& |f|^2 + \left(|\x||\nabla f|\right)^2 + 2
   \rm{Re}[\overline{f}(\x \cdot \nabla) f]
\end{eqnarray*}
and, on integration by parts, initially for $f \in
C_0^{\infty}(\R^n)$ and then by the usual continuity argument,
\begin{eqnarray*}
\int_{\R^n} \overline{f}(\x\cdot \grad)f d\x&=& \sum_{j=1}^n
\int_{\R^n} x_j \overline{f}\frac{\partial f}{\partial
x_j}d\x \\
&=& -\sum_{j=1}^n \int_{\R^n} f \left\{\overline{f}+
x_j\frac{\partial \overline{f}}{\partial x_j}\right \}d\x\\
&=& - \int_{\R^n} \left \{n|f|^2 +f (\x \cdot
\grad)\overline{f}\right \} d\x.
\end{eqnarray*}
This gives
\[ 2\rm{Re}
\int_{\R^n}[\overline{f}(\x \cdot \nabla) f]d\x = -n
\int_{\R^n}|f|^2 d\x
\]
and hence
\begin{eqnarray}\label{modHardy}
 \int_{\R^n}|\nabla(|\x|f)|^2 d\x &=& \int_{\R^n}\left(|\x||\grad
 f|\right)^2d\x - (n-1)\int_{\R^n}|f|^2d\x \nonumber \\
& \ge &  \int_{\R^n} | Lf|^2d\x - (n-1)\int_{\R^n}|f|^2d\x
\end{eqnarray}
which confirms our claim. On substituting (\ref{modHardy}) and
$f=h/|\cdot|$ in Corollary 4 we get
\begin{eqnarray*}
\|\mathcal{M}(h)\|^2_{L^{2^{*}}(\R^+;d\mu)}&\le & C\left\{\|\grad
h \|^2_{L^2(\R^n)} + (n-1) \|h/|\cdot|\|^2_{L^2(\R^n)} \right.\\
&-& \left.(n^2/4) \|h/|\cdot|\|^2_{L^2(\R^n)}\right
\}^{1/n}\|h/|\cdot|\|_{L^2(\R^n)}^{2(1-1/n)}
\end{eqnarray*}
which yields (\ref{NewHardy}). The inequality (\ref{NewHardy2})
follows from
\[
n [\varepsilon/(n-1)]^{1-1/n} ab \le a^n + \varepsilon b^{n/(n-1)}
\]
which is a consequence of Young's inequality.
\end{proof}

\bigskip

The inequality (\ref{NewHardy2}) is implied by Stubbe's inequality
(\ref{Stubbe}). For on setting $\delta = (n-2)^2/4 - \varepsilon $
in (\ref{NewHardy2}) we have
\begin{equation}\label{Stubbe0}
\|\mathcal{M}(h)\|^2_{L^{2^{*}}(\R^+;d\mu)}\le  C
[\frac{(n-2)^2}{4}-\delta]^{-\frac{(n-1)}{n}} \{\|\grad h
\|^2_{L^2(\R^n)} - \delta \|h/|\cdot|\|^2_{L^2(\R^n)} \}.
\end{equation}
Since
\[
\|\mathcal{M}(h)\|^{2^{*}}_{L^{2^{*}}(\R^+;d\mu)} \le
\frac{1}{|\mathbb{S}^{n-1}|} \|h\|^{2^{*}}_{L^{2^{*}}(\R^n)}
\]
by H\"{o}lder's inequality, it follows that (\ref{Stubbe0}) is a
consequence of (\ref{Stubbe}).

\bigskip

If in (4.40) $g=\Phi f,$ where $f$ is supported in the annulus
$A(1/R,R) := \{\x \in \R^n: 1/R \le |\x| \le R\},$ then $G$ is
supported in the interval $[-\ln R,\ln R]$ and we have

\begin{Corollary}\label{compactSobmod}
Let $f$ in Corollary 4 be supported in the annulus $A(1/R,R).$
Then
\begin{equation}\label{penult2}
    \|r\mathcal{M}(f)(r)\|^2_{L^{2^{*}}(\R^+;d\mu)} \le C
    (\ln R)^{\frac{2(n-1)}{n}}\left \{\|Lf \|^2_{L^2(\R^n)}
    -\frac{n^2}{4}\|f\|^2_{L^2(\R^n)}\right \}.
\end{equation}
\end{Corollary}
On putting $f=h/|\cdot|$ in (\ref{penult2}) we have as in the
proof of Corollary ~5
\begin{Corollary}\label{compactHardy}
Let $h$ in Corollary 5 have support in the annulus $A(1/R,R).$
Then
\begin{equation}\label{penult1}
    \|\mathcal{M}(h)\|^2_{L^{2^{*}}(\R^+;d\mu)} \le C
    (\ln R)^{\frac{2(n-1)}{n}}\left \{\|\grad h \|^2_{L^2(\R^n)}
    -\frac{(n-2)^2}{4}\|\frac{h}{|\cdot|}\|^2_{L^2(\R^n)}\right \}.
\end{equation}
\end{Corollary}

Finally we have the following $p=2$ case of Corollary 3.

\begin{Corollary}
Let $2<q<\infty $ and $m=q/2-1.$ Then, if $f$ is such that $f, Lf
\in L^2(\R^n)$ and
$\int_{\R^+}\int_{\mathbb{S}^{n-1}}|f(s\omega)|^m s^{n(\frac
{m}{2}-1)}ds d\omega < \infty,$ we have that
$\int_{\R^+}|f(s\omega)|^q s^{nm}ds d\omega < \infty $ and
\begin{eqnarray}
 \int_{\R^+}\int_{\mathbb{S}^{n-1}}|f(s\omega)|^q s^{nm}ds d\omega &\le &  C \left\{
\|Lf\|^2_{L^2(\R^n)}
    -\frac{n^2}{4}\|f\|^2_{L^2(\R^n)}\right \}^2  \nonumber \\
    & \times & \left \{\int_{\R^+}\int_{\mathbb{S}^{n-1}}|f(s\omega)|^m s^{n(\frac {m}{2}-1)}ds d\omega
    \right\}^2 \nonumber \\
\end{eqnarray}
\end{Corollary}
\begin{proof}
Corollary \ref{GagNir} with $p=2$ yields
\begin{eqnarray*}
\|\mathcal{M}(\Phi f)\|_{L^q(\R)}&\le &  C \left\{
\|Lf\|^2_{L^2(\R^n)}
    -\frac{n^2}{4}\|f\|^2_{L^2(\R^n)}\right \}^{2/q}  \nonumber \\
    & \times& \|\Phi f\|^{1-2/q}_{L^m(\R)}.
\end{eqnarray*}
Since
\[
\|\mathcal{M}(\Phi f)\|_{L^q(\R)} =
|\mathbb{S}^{n-1}|^{-1}\|s^{nm} f\|_{L^{q}(\R \times
\mathbb{S}^{n-1})}
\]
and
\[
\|\Phi f\|^m_{L^m(\R\times \mathbb{S}^{n-1})} =
\int_{\R^+}\int_{\mathbb{S}^{n-1}}|f(s)|^m s^{n(\frac {m}{2} -1)}
ds d\omega
\]
the corollary follows.
\end{proof}

\bibliographystyle{amsalpha}

\end{document}